\documentclass[leqno,12pt]{amsart}

\usepackage{amsmath,amsfonts,amsthm,amsopn,cite,mathrsfs}
\usepackage{amssymb,latexsym,amsmath,amscd,epsfig}
\theoremstyle{theorem}
\newtheorem{lemma}{Lemma}
\newtheorem{theorem}{Theorem}
\newtheorem*{corollary}{Corollary}

\newtheorem{proposition}{Proposition}

\theoremstyle{remark}

\newtheorem*{remark}{Remark}

\newcommand{\C}{\field{C}}
\newcommand{\Z}{\field{Z}}

\newcommand{\N}{\field{N}}
\newcommand{\R}{\field{R}}
\newcommand{\card}{{\rm card}}

\newcommand{\rank}{{\rm rank}}

\newcommand{\tr}{{\rm tr}}

\newcommand{\Rr}{{\mathcal R}}
\newcommand{\n}{{\mathcal N}}
\newcommand{\F}{{\mathcal F}}
\newcommand{\M}{{\mathcal M}}

\newcommand{\beq}{\begin{equation}}
\newcommand{\eeq}{\end{equation}}
\newcommand{\beqs}{\begin{equation*}}
\newcommand{\eeqs}{\end{equation*}}
\newcommand{\beg}{\begin{gather}}
\newcommand{\eeg}{\end{gather}}
\newcommand{\ora}{\overrightarrow}
\newcommand{\field}[1]{\mathbb{#1}}
\newcommand{\bR}{\field{R}}        
\newcommand{\bN}{\field{N}}        
\newcommand{\bZ}{\field{Z}}        
\newcommand{\bC}{\field{C}}        
\newcommand{\up}{uncertainty principle}
\newcommand{\tfa}{time-frequency analysis}

\newcommand{\stft}{short-time Fourier transform}

\newcommand{\tf}{phase-space}

\newcommand{\fif}{if and only if}
\newcommand{\tfs}{phase-space shift}

\newcommand{\onb}{orthonormal basis}

\newcommand{\rb}{Riesz basis}

\newcommand{\modsp}{modulation space}

\newcommand{\rem}{\noindent\textsl{REMARK:}}
\newcommand{\vs}{\vspace{3 mm}}

\def\inv{^{-1}}
\def\rd{\bR^d}
\def\cd{\bC^d}
\def\zd{\bZ^d}
\def\rdd{{\bR^{2d}}}
\def\zdd{{\bZ^{2d}}}
\def\lrd{L^2(\rd)}
\def\intrd{\int_{\rd}}
\def\intrdd{\int_{\rdd}}

 \def\cF{\mathcal{F}}              

 \def\cB{\mathcal{B}}

\begin{document}
\begin{abstract}
We prove a strong \up\ for Riesz bases in $\lrd $ and show that the
\onb\ constructed by Bourgain possesses the optimal \tf\
localization. 
\end{abstract}

\title[ Localization  properties of bases for
$L^2(\R^d)$]{Phase Space  Localization of Riesz bases for $L^2(\R^d)$} 
\author{Karlheinz Gr\"ochenig}
\address{Faculty of Mathematics \\
University of Vienna \\
Nordbergstrasse 15 \\
A-1090 Vienna, Austria}
\email{karlheinz.groechenig@univie.ac.at}
\author{Eugenia Malinnikova}
\address{Department of Mathematics,
Norwegian University of Science and Technology,
7491, Trondheim, Norway}
\email{eugenia@math.ntnu.no}

\date{}
\thanks{K.\ G.\ was
  supported in part by the  project P22746-N13  of the
Austrian Science Foundation (FWF). E.\ M.\ was supported by the
Research Council of Norway grant 185359/V30.}
\maketitle


\section{Introduction}
In ~\cite{B} J.\ Bourgain constructed an  orthonormal basis for
$L^2(\bR )$ consisting of functions $f_n \in L^2(\bR )$, such that 
\begin{equation}
  \label{eq:c16}
\sup _{n\in \bN }\Big( \inf_{a\in \bR } \int_{\R}|x-a|^{2}|f_n(x)|^2dx+ 
\inf_{b\in \bR } \int_{\R}|\xi-b|^{2}|\widehat{f_n}(\xi)|^2d\xi\Big)<\infty\, .
\end{equation}
Bourgain remarked that the exponent $2$ of $|x-a|$ and $|\xi -b|$ is
optimal and that there are no orthonormal bases with a better
\tf\  localization.

In this paper we prove the following strong \up\ for Riesz bases for
$\lrd $.  
\begin{theorem}
\label{th:1}
If  $\{f_n\}_{n=1}^\infty$ is a Riesz basis for $L^2(\bR^d)$ and $s>d$, then
\begin{equation}
  \label{eq:c1}
\sup _{n\in \bN }\Big(\inf_{a\in \rd}\int_{\bR^d}|x-a|^{2s}|f_n(x)|^2dx+ 
\inf_{b\in \rd } \int_{\bR^d}|\xi-b|^{2s}|\widehat{f_n}(\xi)|^2d\xi \Big) =\infty\, .  
\end{equation}
\end{theorem}
This theorem therefore asserts that the Bourgain basis possesses the
best possible  \tf\ localization.
For the case of an orthonormal basis for $L^2(\bR)$ in dimension $d=1$,
Bourgain outlines  a proof strategy for  Theorem~\ref{th:1}.
Precisely, he writes that  
{\slshape{"it has been shown by T.~Steger that $L^2(\bR)$ does not admit
    a basis of the form $f_j=e^{ib_jx} g_j(x-a_j),$ where $g_j$
    satisfies $\sup_j\|g_j\|_{A_\epsilon}<\infty,$ defining 
\[
\|g\|^2_{A_\epsilon}=\int(1+x^2)^{1+\epsilon}|g_j(x)|^2dx+\int(1+\xi^2)^{1+\epsilon}|\widehat{g_j}(\xi)|^2d\xi\, .  
\]
Here $\epsilon>0$ is any strictly positive number. His argument is
based on the fact that the operations $x$ ($x$-multiplication) and
$d/dx$ in the latter basis would become "almost" diagonal operators,
violating the non-commutation property $[d/dx, x]=I$. He also makes
use of a density computation due to Y.~Meyer of the set $\Lambda$ of
pairs $(a_j, b_j)$ in phase space. The condition $\epsilon>0$ is
important in Steger's argument as well as for Meyer's distribution
result to be valid"}}, see \cite{B}.

Some of these arguments have made  their way into the literature.
  A density argument related to Meyer's argument  has appeared in the
fundamental paper of Ramanathan and Steger~\cite{RS} on the density of
Gabor frames and has become the main technique to investigate the
density of frames. See ~\cite{BCHL06a,GR96a,heil07} for some
variations of the Ramanathan-Steger technique. 
The canonical commutation relations were used in Battle's elegant proof of
the Balian-Low theorem~\cite{Bat}.   

However, a full proof of the \up\ of Theorem~\ref{th:1} has not yet been
given. Research has focused mainly on  bases  consisting of
phase-space shifts $f_n(x)= e^{2\pi i b_n x} g(x-a_n)$ of  a single generating function
$g$, so-called Gabor systems.   The theorem of Balian-Low  asserts
that  a basis satisfying~\eqref{eq:c16}  cannot  consist of  a (regular) Gabor
system. We refer the reader to proofs of the theorem in \cite{D} and
\cite{Bat}, to the survey articles on the  Balian-Low theorem and its
generalizations \cite{BHW,CP} and to the monograph \cite{G1} for
detailed discussions of the subject.   Gabor
systems are somewhat easier to handle, because one needs to control
the localization of only one function in contrast to  Bourgain's
case. 
  
In this paper we offer a complete proof of Theorem~\ref{th:1} which  extends the result mentioned in \cite{B}
to higher dimensions and to Riesz bases instead of orthonormal bases. For
orthogonal bases  our proof
follows the outline  of Bourgain.  The case of Riesz bases requires
additional ideas. We will  apply  the theory of localized
frames~\cite{FoG05,G2} to verify that the biorthogonal basis possesses
the same localization properties as the original basis. In a second
step we use a bootstrap argument. We will show that if a Riesz basis
violates condition~\eqref{eq:c1} for $s>d$, then we can construct a new Riesz
basis with optimal \tf\ localization, for instance, with all functions
in a Gelfand-Shilov space of test functions. 

It may seem a lot of effort to prove the non-existence of
well-localized \tf\ bases, but several  arguments are   of interest in
themselves. The proof combines tools from  the density theory of
frames, the canonical commutation relations, the theory of localized
frames, recent \tf\ methods,  and a new argument of how to improve the quality of a given
basis.

One of the corollaries of Theorem~\ref{th:1} is that there is no Riesz basis
of phase-space shifts of the Gaussian function in
$L^2(\bR^d)$. This fact 
implies that there is no subset  $\Lambda\subset  \C^d$ that is  both
 sampling and interpolating for the Bargmann-Fock space $
{\mathcal{F}}^2(\C^d)$.   This statement is well-known in
dimension $d=1$, but seems to have been open in higher dimensions. 

The paper is organized as follows: In Section~2 we show that a
well-localized \tf\ basis must be indexed by a set of density one. In
Section~3 we prove Theorem~\ref{th:1} for the special case when its
biorthogonal basis  is also well-localized, it includes the case of an 
orthonormal basis. Section~4 contains some preparations from \tfa . In
Section~5 we develop the necessary arguments to  prove the uncertainty
principle of Theorem~\ref{th:1} for  Riesz bases. Section~6 elaborates
the non-existence of sets of simultaneous  sampling and interpolation
and concludes with further remarks.


\section{Density Conditions}


We say that a sequence of functions $\{f_n\}_{n=1}^\infty \subset
L^2(\R^d)$ has phase-space localization of magnitude $s$, if  
\[
\sup _{n\in \bN }\Big( \inf_{a\in \rd } \int_{\bR^d}|x-a|^{2s}|f_n(x)|^2dx+ 
\inf_{b\in \rd } \int_{\bR^d}|\xi-b|^{2s}|\widehat{f_n}(\xi)|^2d\xi\Big) <\infty\, .
\]
In this case there exist points $(a_n,b_n) \in \rdd $, such that
\[\sup _{n\in \bN }\Big(\int_{\bR^d}|x-a_n|^{2s}|f_n(x)|^2dx+ 
\int_{\bR^d}|\xi-b_n|^{2s}|\widehat{f_n}(\xi)|^2d\xi\Big) <\infty\, .\]
Then the set $\Lambda = \{ (a_n, b_n)\}_{n=1}^\infty $ is the set in
the  phase-space where the  functions $\{f_n\}_n$ are localized. Note
that there is some freedom in the choice of points $(a_n,b_n)\in
\rdd$. 

We will first  estimate the density of the set
$\Lambda=\{(a_n,b_n)\}_{ n=1}^\infty \subset\bR^{2d}$ both for Riesz bases and frames
for $L^2(\bR ^d)$ which have phase-space localization. The ideas we follow are well known, see
\cite{RS, SP, LP, S}.   
 
Let $\Lambda$ be a  subset of $\bR^{2d}$, we denote by
$D^{+}(\Lambda)$ and $D^{-}(\Lambda)$ its upper and lower Beurling
densities,  
\[
D^+(\Lambda)=\limsup_{r\rightarrow\infty}\sup_{x\in\bR^{2d}}\frac{\card(\lambda\cap
  Q(x,r))}{|Q(x,r)|},\quad 
D^-(\Lambda)=\liminf_{r\rightarrow\infty}\inf_{x\in\bR^{2d}}\frac{\card(\lambda\cap
  Q(x,r))}{|Q(x,r)|},\] 
where $x=(x_1,x_2)\in\bR^d\times\bR^d$ and  
$$ 
Q(x,r)=\{(y_1,y_2)\in\bR^d\times\bR^d: |x_1-y_1|<r, |x_2-y_2|<r\} \, .
$$
 These densities can be also
defined by using dilations of cubes or balls in $\bR^{2d}$ instead of
$Q(x,r)$, as was proved by Landau~\cite{L}. 

 A set $\Lambda \subset \rdd  $ is relatively separated, if $\sup _{x\in
  \rdd } \card \big(\Lambda \cap (x+[0,1]^{2d})\big) <\infty $.
Clearly, if $D^+(\Lambda) <\infty $, then $\Lambda $ is relatively
separated. 
  
\begin{lemma}
\label{l:1}
Suppose that $\{f_n\}_{n=1}^\infty$ is a Riesz basis for $L^2(\bR^d)$ that has phase-space localization of magnitude $s$, $s>0$, at points $\{(a_n,b_n)\}_{n=1}^\infty$, i.e.,
\[
\sup _{n\in\bN}\Big( \intrd |x-a_n|^{2s}|f_n(x)|^2dx + 
\intrd |\xi-b_n|^{2s}|\widehat{f_n}(\xi)|^2d\xi \Big) =  S <\infty \, .
\]
Then $\Lambda=\{(a_n,b_n)\}_{n=1}^\infty \subset \rdd$ is  relatively
separated  and $D^+(\Lambda)\le 1$. 
\end{lemma}

\begin{proof}
Fix $\epsilon >0$. We say that a function $g\in L^2(\bR^d)$ is $\epsilon$-concentrated on some set $E\subset\bR^d$ if
\[
\int_E|g(x)|^2dx\ge(1-\epsilon^2)\|g\|^2.\]
 
Since 
$$
\int _{|x-a_n|\geq r} |f_n(x)|^2 \, dx \leq r^{-2s} \int _{|x-a_n|\geq r}
|x-a_n|^{2s} |f_n(x)|^2 \, dx \leq r^{-2s} S \, ,
$$
 there exists $r=r(\epsilon )$ such that $f_n$ is
$\epsilon$-concentrated on $B(a_n,r)$ uniformly in $n$. Likewise $\widehat{f_n}$ is
$\epsilon$-concentrated on $B(b_n,r)$ for every $n$. We fix $(x_0,\xi_0)\in\bR^d\times\bR^d$, consider any $R>0$ and denote $Q_R=Q((x_0,\xi_0),R)$. Remark that if
$ (a_n,b_n)\in Q_R$, then $f_n$ is  $\epsilon$-concentrated
on $B(x_0,R+r)$,  and $\widehat{f_n}$ is 
$\epsilon$-concentrated on $B(\xi_0,R+r)$, where $r=r(\epsilon)$ as above. 

We now apply a standard
estimate of the trace of a time-frequency restriction  operator to
conclude that $D^+(\Lambda)\le 1$, see \cite{RS}. 

Let  $\cF$ be  the Fourier transform and $P_E$ be the  projection
operator $P_E f = \chi _E \, f$ (multiplication of $f$  by the characteristic function of
$E$). The  \tf\ restriction operator is defined by 
\[L=P_{B(x_0,R+r)} (\F^{-1} P_{B(\xi_0,R+r)}\F) P_{B(x_0,R+r)}=P_1 P_2 
P_1\, .\]
  It is well-known, see for example~\cite{FS}, that  
\[\tr(L)=|B(x_0,R+r)||B(\xi_0,R+r)|=|Q_{R+r}|.\]
For each $f_n$ such that $(a_n,b_n)\in Q_R$,  we have 
\[\|f_n-Lf_n\|\le\|f_n-P_1f_n\|+\|P_1\|\|f_n-P_2 f_n\|+\|P_1P_2\|\|f_n-P_1f_n\|\le
3\epsilon\|f_n\|.\]
Now let $\{g_n\}$ be the biorthogonal basis for $\{f_n\}$, i.e.,   $(f_n,g_m)=\delta_{nm}$. Then   
\[
  \tr(L) \ge \sum_{(a_n,b_n)\in Q_R}(Lf_n,g_n) 
 \ge \sum_{(a_n,b_n)\in Q_r}\left((f_n,g_n)-|(f_n-Lf_n,g_n)|\right) 
\ge (1-3C\epsilon)\card(\Lambda\cap Q_R),
\]
where $C=\sup_n\|f_n\|\|g_n\|<\infty$ (since $\{f_n\}$ is a Riesz basis). Thus  
\[  
\card\, \Big(\Lambda\cap Q((x_0,\xi_0),R)\Big)\le (1-3C\epsilon)^{-1}|Q((x_0,\xi_0),R+r)|.\] 
Taking the limit $R\to \infty $, 
we obtain $D^+(\Lambda)\le (1-3C\epsilon)^{-1}$ for every
$\epsilon>0$,  and thus  $D^+(\Lambda)\le 1$, and $\Lambda $ is
relatively separated. 
\end{proof}

\begin{remark}
If $\{f_n\}_{n=1}^\infty$ is a frame that has phase-space localization of magnitude $s>0$ at points $\{(a_n,b_n)\}_{n=1}^\infty$ and satisfies   $\|f_n\|_2\leq C$,  then it is still true  that
$\Lambda=\{(a_n,b_n)\}_{n=1}^\infty $ is a relatively separated  set and that
$D^+(\Lambda)<\infty$. This follows by compactness arguments, see Theorem 3.5  in \cite{JP} for a similar result in dimension $d=1$. 
\end{remark}

\begin{lemma}
\label{l:2}
Suppose that  $\{f_n\}_{n=1}^\infty$ is a frame for $L^2(\bR^d)$ with 
  $\sup _{n\in\bN}\|f_n\|_2 =
C < \infty $. If  $s>d$ and $\{f_n\}_n$ has phase-space localization of magnitude $s$ at points $\{(a_n,b_n)\}_n$, then  $\Lambda=\{(a_n,b_n)\}_{n=1}^\infty$ is relatively separated and $D^-(\Lambda)\ge 1$.  
\end{lemma}

We remark that the Lemma does not hold for $s=d$. This can be seen
from the construction of an  orthonormal basis in \cite{B}.  

\begin{proof}
Let  $K(y,l)$ denote the  cube with center $y\in \bR^q$ and the side length $2l$,
\[ K(y,l)=\{z\in \bR^q:\|y-z\|_\infty<l\},\]
where $\|z\|_\infty=\max_{1\le s\le q} |z_s|,\ z=(z_1, \dots , z_q)\in\bR^q$.
 
Fix $\epsilon >0$ and choose   $\delta $ in the open interval  
$(d/s,1)$. This is possible by the hypothesis $s>d$. 

\vs

\textbf{Step 1. An estimate for the coefficients $(\psi , f_m)$ of a
  localized function.} 
  Assume
that $\psi\in L^2(\bR^d)$, $\|\psi\|_2=1$, $\psi$ is 
$\epsilon$-concentrated on $K(a,R-R^{\delta})$  and  its Fourier
transform is supported on $K(b,R-R^{\delta})$. Set
$\eta=\psi(1-\chi_{K(a,R-R^{\delta})})$, so that  $\|\eta\|_2\le
\epsilon$.
 
If $\|a_m-a\|_\infty>2^k R$,  then the following estimate holds:
\begin{align*}
|(\psi,f_m)|^2 & \le
2|(\psi(1-\chi_{K(a,R-R^{\delta})}),f_m)|^2+2\left(\int_{K(a,R-R^{\delta})}|\psi(x)||f_m(x)|dx\right)^2
\\ 
& \le
2|(\eta,f_m)|^2+2\left(((2^k-1)R+R^{\delta})^{-s}\int_{K(a,R-R^{\delta})}|x-a_m|^{s}|\psi(x)||f_m(x)|dx\right)^2
\\
 & \le 2|(\eta,f_m)|^2+2((2^k-1)R+R^{\delta})^{-2s}S\, .
\end{align*}
If  $\|b_m-b\|_\infty>2^k R$, then 
\[
|(\psi,f_m)|^2  =|(\widehat{\psi},\widehat{f_m})|^2=\left(\int_{K(b,R-R^\delta)}|\widehat{\psi}(\xi)||\widehat{f_n}(\xi)|d\xi\right)^2\le ((2^k-1)R+R^\delta)^{-2s}S\, . 
\]
Let $\M_0=\{n: (a_n,b_n) \in K (a,R)\times K(b,R)\}$ and $\M$  be the complement of $\M _0$, $\M=\{n:
(a_n,b_n)\not\in K (a,R)\times K(b,R)\}$. We further partition  $\M$ into the sets $\M_k$ as follows:
\[
\M_k=\{n: \max(\|a_n-a\|_{\infty}, \|b_n-b\|_{\infty})\in [2^{k}R,2^{k+1}R)\},\quad k\ge 0.\] 
Since  $D^+(\Lambda)<\infty$ by Lemma~\ref{l:1} (see also the remark after the lemma),
we find that  $\card(\M_k)\le
C_1 (2^kR)^{2d}$ for some constant $C_1$ large enough.
Thus  
\begin{align*}
\sum_{m\in\M}|(\psi,f_m)|^2& =\sum_{k=0}^\infty
\sum_{m\in\M_k}|(\psi ,f_m)|^2 \\
& \le 2 \sum _{m\in \M } |(\eta , f_m)|^2 +
2 S \sum_{k=0}^\infty C_1 2^{2kd}R^{2d}\left(\left(2^{k}-1\right)R+R^{\delta}\right)^{-2s}
 \\
& \le 2 B  \|\eta \|_2^2 + 2SC_1R^{2d-2s\delta}+ 2SC_1 R^{2d-2s} \sum _{k=1}^\infty
2^{2kd}(2^k-1)^{-2s} \, , 
\end{align*}
where $B$ is the upper frame bound of $\{f_n\}_n$.
Since  $s>d$ by assumption, the last sum converges.  Further, $s\delta> d$ and, by choosing
$R$ large enough, the second and third terms can be made arbitrarily small. Given
$\epsilon >0$ and $\delta\in(d/s,1)$, we find that
\begin{equation}
  \label{eq:c13}
\sum_{m\in\M}|(\psi,f_m)|^2 \leq C_0^2 \epsilon ^2 \qquad   \text{ for }
R\geq R_0(\epsilon, \delta, \Lambda, S) \, ,
\end{equation}
with the  constant $C_0$  depending  only on the frame bound $B$ of $\{f_n\}_n$. 


\vs

\textbf{Step 2. Comparison with a basis of  prolate spheroidal
  functions.} 
For given  $\epsilon >0$,  $\delta \in (d/s  , 1) $,   and $R\geq
R_0(\epsilon, \delta, \Lambda, S)$, we now  consider those   prolate spheroidal
 functions $\phi_1,...,\phi_N$ with  $N= N(R)$  that are $\epsilon
d^{-1/2}$ concentrated on $(-R+R^{\delta},R-R^{\delta})$ and whose
Fourier transforms are supported on $(-R+R^{\delta},R-R^{\delta})$. We
refer the reader to  \cite{SP} and \cite{LP} for definitions and
properties of these functions. According to ~\cite{LP} the number of
$\phi _j$ with these concentration properties satisfies  
$\lim_{R\rightarrow \infty}N(R)R^{-2}=1$. 
 
In higher dimensions we  take  tensor products of phase-space
shifts of these prolate spheroidal functions. Let $\sigma = (n_1, \dots , n_d) \in \{ 1, 2, \dots ,
N(R)\}^d$ and define 
\[\psi_{\sigma}(x)=\prod_{j=1}^d e^{-2\pi i
  b_j x_j}\phi_{n_j}(x_j-a_j),\]
then we obtain an orthonormal set of $N^d$ functions $\{\psi _\sigma \}_\sigma$  that are 
 $\epsilon$-concentrated on $K(a,R-R^{\delta})$ and whose Fourier
 transforms are  supported on $K(b,R-R^{\delta})$.  

Now let $\{g_n\}_n$ be the dual frame of $\{f_n\}_n$. If $A>0$ is the
lower frame bound of $\{f_n\}_n$, then we have 
$$
\|\sum _n c_n g_n \|_2^2 \leq A\inv \|c\|_2^2 \qquad \text{ for every
} c\in \ell ^2 \, .
$$

\textbf{Step 3. Density estimate. } 
We now follow the  argument of  Ramanathan and
Steger in~\cite{RS}. Let $S$ be the orthogonal projection of
$L^2(\bR^d)$ onto $\Psi=span\{\psi_{\sigma}: \sigma \in \{ 1 , \dots , N(R)\}^d\}$ and $T$ be the
orthogonal projection onto $G=span\{g_n: n\in \M_0\}$. We consider
$U:\Psi\rightarrow\Psi$, $U=S\circ T$. For each $\psi  \in\Psi$ we obtain   
\begin{align*}
\| \psi  -U\psi \|_2 & =\|S(\psi -T\psi )\|_2\le \|\psi -T\psi \| 
 =  \inf _{g\in G} \|\psi  - g\| \le \|\psi  -\sum_{n\in
  \M _0}(\psi ,f_n)g_n\| \\
&= \|\sum_{m\in \M}(\psi ,f_m)g_m\| 
\le A^{-1/2} \Big(\sum_{m\in \M}|(\psi ,f_m)|^2\Big)^{1/2} \, . 
\end{align*}
Since each basis function  $\psi _\sigma $ is in $ \Psi $ and   satisfies the concentration assumptions from Step~1, the  estimate  \eqref{eq:c13}, implies that 
$$
\|\psi _\sigma - U\psi _\sigma \|_2 \leq A^{-1/2} C_0 \epsilon  \, .
$$
Consequently,  
\[\tr(U)\ge \sum_\sigma (U\psi_\sigma,\psi_\sigma) = \sum_\sigma
\Big(\|\psi _\sigma \|_2^2 -  (\psi
_\sigma - U\psi_\sigma,\psi_\sigma)\Big) \ge (1-A^{-1/2}C_0 \epsilon) N(R)^d\, .\] 
On the other hand,  since $U$
is the composition of two projections,  all eigenvalues of $U$ belong
to $(0,1)$, and therefore  $\tr (U)\le \rank(U)\le \dim
(G)$. Thus  
\[(1-A^{-1/2}C_0\epsilon)N(R)^d\leq \tr(U) \leq \card(\Lambda\cap K(a,R)\times K(b, R))\, .\]

We now use the definition of the Beurling density with cubes in $\rdd
$ instead of balls, and obtain
\begin{align*}
D^-(\Lambda ) &= \lim _{R \to \infty } \inf _{(a,b)\in \rdd }
\frac{\card\big(\Lambda\cap K(a,R)\times K(b, R)\big)}{R^{2d}} \\
& \ge  (1-A^{-1/2}C_0\epsilon) \lim _{R\to \infty } \frac{N(R)^d}{R^{2d}} =
  1-A^{-1/2}C_0 \epsilon \, .  
\end{align*}
As  $\epsilon >0$ was arbitrary, we conclude that $D^-(\Lambda ) \geq
1$. 
\end{proof}

Combining Lemmas~\ref{l:1} and~\ref{l:2}, we obtain the 
density result for localized Riesz bases (recall, however,  that our
aim is to prove that there are no such bases).  
 
\begin{corollary}
  If $s>d$ and $\{f_n\}_{n=1}^\infty$ is a \rb\ for $L^2(\bR^d)$ that has phase-space localization of magnitude $s$ at points $\{(a_n,b_n)\}_{n=1}^\infty$, then the density of $\Lambda=\{(a_n,b_n)\}_n $ is
$D(\Lambda) =  D^+(\Lambda )  = D^-(\Lambda ) =  1$. 
\end{corollary}

\section{Uncertainty identity}
We first prove Theorem \ref{th:1}  under the  additional condition
that the dual basis is also well-localized.  
The proof extends  Battle's elegant proof of the Balian-Low theorem
\cite{Bat} and rediscovers  Steger's argument mentioned by Bourgain in \cite{B} (see the quote above).

The core of the argument
is the  following uncertainty identity  (the canonical commutation
relations) 
\[(xf,\nabla g)+(\nabla f,xg)= \sum _{j=1}^d (x_j f , \tfrac{\partial
  g}{\partial x_j}) + (\tfrac{\partial
  f}{\partial x_j}, x_j g  ) =  -d(f,g),\] 
which holds provided that $f,g,\tfrac{\partial
  f}{\partial x_j}, \tfrac{\partial
  g}{\partial x_j},  x_jg, x_jf\in L^2(\bR^d)$ for $j=1,\dots , d$.

\begin{lemma} 
\label{l:un}
Assume that $\{f_n\}_{n=1}^\infty$ is  a Riesz basis  for $L^2(\bR^d)$ with the
biorthogonal basis  $\{g_n\}_{n=1}^\infty$. If the bases satisfy the localization
estimates 
\begin{itemize}
\item[(a)]
$\int_{\R^d}|x-a_n|^{2s}|f_n(x)|^2dx+\int_{\R^d}|\xi-b_n|^{2s}|\widehat{f_n}(\xi)|^2d\xi\le
S^2 <\infty$ for every $n$;
\item[(b)]
$\int_{\R^d}|x-a_n|^{2s}|g_n(x)|^2dx+
\int_{\R^d}|\xi-b_n|^{2s}|\widehat{g_n}(\xi)|^2d\xi\le T^2  <\infty$ for every  $n$;
and 
\item[(c)]
 $\Lambda =\{(a_n,b_n)\}_{n=1}^\infty\subset\rdd$  is  relatively
 separated and
 $0<D^-(\Lambda)\le D^+(\Lambda)<\infty$, 
 \end{itemize}
 then $s\le d$.
 \end{lemma}

\begin{proof}
We assume that $s>d$ and use the uncertainty identity to derive a contradiction from $(a)-(c)$.
In the following we will write $\ora{(xf, g)}\in \cd $ for the vector with
components $(x_j f , g), j=1, \dots , d$. Likewise $\ora{(\nabla f , g)} = (\tfrac{\partial
  f}{\partial x_j}, g)_{j=1}^d$.  

\vs

\textbf{Step 1. An estimate for non-diagonal coefficients.} Condition
$(a)$ implies  $xf_n\in L^2(\bR^d)^d$,  and then the sequence of 
vectors 
\[
c_m^n=\ora{(xf_n, g_m)}=(x_jf_n, g_m)_{j=1}^d\in \C^d\]
is  well defined.
By the  biorthogonality condition,  for $m\neq n$ 
\[c_m^n=\ora{((x-a_n)f_n,g_m)}.\] Since $\{g_m\}$ is a Riesz basis, 
assumption (b) implies that 
\[
\sum_{m:m\neq n}|c_m^n|^2\le B\| |x-a_n|\, f_n\|^2_2\le B S^2 \, .\]
Next, since  $\overline{c_m^n}=\ora{(xg_m,f_n)}$, we also have  
\[
\sum_{n:n\neq m}|c_m^n|^2\le B \| |x-a_m|\, g_m\|^2_2\le BT^2.
\] 
Here $B$ is  the  upper basis constant  for both Riesz bases $\{g_m\}_m$ and
$\{f_n\}_n$. 

The coefficients 
 \[
 d_m^n=\ora{(\zeta\widehat{f_n},\widehat{g_m})}=(2\pi i)^{-1}\ora{(\nabla f_n,g_m)}\]
 enjoy  similar properties.
  
\vs

\textbf{Step 2. Commutation relations.} 
We now apply  the uncertainty identity to each pair $\{f_n, g_n\}$ and
obtain 
\begin{align}
\label{eq:h}
d& =-\sum_m\Big(\ora{(xf_n,g_m)}\cdot\ora{(\nabla g_n,f_m)}+\ora{(\nabla f_n,g_m)}\cdot\ora{(xg_n,f_m)}\Big)\\
& =- 2\pi i\sum_m(
c_m^n\cdot\overline{d_n^m}-d_m^n\cdot\overline{c_n^m}), \notag
\end{align} 
where $\lambda\cdot \mu=\sum_{j=1}^d\lambda_j\overline{\mu_j}$ is the standard scalar product in $\C^d$.

For each  $R>0$ define $\n(R)=\{n: |a_n|\le R, |b_n|\le R\}$ and
$N(R)=\card \, \n(R) $.  Now we sum up the identities (\ref{eq:h}) for all $n\in \n(R)$,
\begin{align}
\frac{d}{2\pi i} N(R) &= \sum_{n\in\n(R)}\sum_m(-c_m^n\cdot\overline{
  d_n^m}+d_m^n\cdot \overline{c_n^m}) \notag \\
&= \sum_{n,m\in\n(R)}\sum_{j=1}^d ( -(c_m^n)_j(d_n^m)_j+(d_m^n)_j(c_n^m)_j)+
\sum_{n\in\n(R)}\sum_{m\not\in\n(R)}(
-c_m^n\cdot\overline{d_n^m}+d_m^n\cdot\overline{c_n^m}). \label{eq:s}
\end{align}
Clearly, the first sum equals zero. We will derive a contradiction by
showing that the second sum  possesses a slower
growth than $N(R) $. We divide the necessary estimates into several steps. 

\vs

\textbf{Step 3. Points $(a_n, b_n)$ near the boundary.} To  estimate
the second sum,  we partition $\n(R)$ into two 
sets, \[\n(R)=\n(R-r)\cup \big(\n(R)\setminus \n(R-r)\big) 
\, , \]
where  $r=
R^\delta$ for some $\delta \in (d/s, 1)$. 

First, for $n\in\Rr(R-r,R)=\n(R)\setminus \n(R-r)$ we get
\begin{align}
\label{eq:Rr}
\sum_{n\in\Rr(R-r,R)} & \sum_{m\not\in\n(R)}|c_m^n||d_n^m|  \le \notag
\\
& \le \sum_{n\in\Rr(R-r,R)}\Big(\sum_{m: m\neq n}|c_m^n|^2\Big)^{1/2}
\Big(\sum_{m: m\neq n}|d_n^m|^2\Big)^{1/2}\le (N(R)-N(R-r))BS_1T_1 \, .
\end{align}
 The sum of $|d_m^n||c_n^m|$ admits the same estimate.

\vs

\textbf{Step 4. Further partition of $\n(R)$ for interior $(a_n,b_n)$. } 
Next  we  partition the complement of $\n(R)$  into the  rings  $\n_k=\n(R_{k+1})\setminus \n(R_{k})$ where $R_k=2^{k}R$, $k\ge 0$. Then for each $n\in\n(R-r)$ we get
\begin{align*}
\sum_{m\not\in\n(R)}|c_m^n||d_n^m|& =
\sum_{k=0}^\infty \sum_{m\in \n_k}|c_m^n||d_n^m| \\
& \le 
\sum_{k=0}^\infty \sum_{m\in \n_k:
  |a_m|>R_{k}}|c_m^n||d_n^m|+\sum_{k=1}^\infty  \sum_{m\in
  \n_k:|b_m|>R_{k}}|c_m^n||d_n^m|\\
& \le 
\Big(\sum_{k=0}^\infty  \sum_{m\in \n _k:
    |a_m|>R_k}|c_m^n|^2\Big)^{1/2}\Big(\sum_{m: m\neq
    n}|d_n^m|^2\Big)^{1/2}+\\ 
& +\Big(\sum_{m: m\neq
    n}|c_m^n|^2\Big)^{1/2}\Big(\sum _{k=0}^\infty \sum_{m\in \n
    _k:|b_m|>R_{k}}|d_n^m|^2\Big)^{1/2}. 
\end{align*}

\vs

\textbf{Step 5. Main estimate.}
Now we write down an estimate for $c_m^n$ when  $|a_n|<R-r$ and
$|a_m|>R_j$.
Set $h_n^{(j)}(x)=(x-a_n)_jf_n(x)(1-\chi_{B(R-r/2)}(x))$, $j=1,2,...,d$. Then
\[
\|h_n^{(j)}\|_2^2\le \int_{|x|>R-r/2}|(x-a_n)_j|^2|f_n(x)|^2dx.\]
Further, for $|x|\geq r/2$ and $|a_n| \leq R-r$, we have $|x-a_n|\geq
r/2$, and therefore 
\beq
\label{eq:hest}
\sum_{j=1}^d\|h_n^{(j)}\|_2^2\le
(r/2)^{2-2s}\int_{\R^d}|x-a_n|^{2s}|f_n(x)|^2dx\le (r/2)^{2-2s}S^2.
\eeq
Then we have
\begin{align}
\label{eq:st5}
|c_m^n|^2 &=\sum_{j=1}^d|\big((x-a_n)_jf_n,g_m\big)|^2  \\
& \leq
\sum_{j=1}^d\left(2|((x-a_n)_jf_n\chi_{B(R-r/2)},g_m)|^2+2|(h_n^{(j)},g_m)|^2\right)
\notag 
\\ 
& \le
2S^2\|g_m\chi_{B(R-r/2)}\|_2^2+2\sum_{j=1}^d|(h_n^{(j)}, g_m)|^2 \notag 
\\
& \le 2S^2T^2(R_k-R+r/2)^{-2s}+2\sum_{j=1}^d|(h_n^{(j)}, g_m)|^2.
\end{align} 
And since $\{g_m\}_m$ is a Riesz basis, 
\beq
\label{eq:st5R}
\sum_m|(h_n^{(j)}, g_m)|^2\le B\|h_n^{(j)}\|_2^2\, . 
\eeq
Summing up the estimates  (\ref{eq:st5}) over all $k$ and all $m\in \n_k$ such that $|a_m|>R_k$ and taking into account (\ref{eq:hest}) and (\ref{eq:st5R}), we obtain
\beq
\label{eq:st5f}
\sum_{k=0}^\infty\sum_{m\in\n_k:|a_m|>R_k}|c_m^n|^2\le
2S^2T^2\sum_{k=0}^\infty
N(R_{k+1})(R_k-R+r/2)^{-2s}+B (r/2)^{2-2s}S^2. 
\eeq 

 We can  derive similar estimates of $\sum_k\sum_{m}|d_n^m|$ where the summation is over all $k$ and  $m\in \n_k$ such that
$|b_m|>R_k$ by using  the localization inequality for
$\widehat{g_n}$. Likewise, we obtain the  estimates for
$\sum_k\sum_{m}|c_n^m|$ and 
$\sum_k\sum_{m}|d_m^n|$, where the sums are  over all $k$
and $m\in \n_k$ such that $|b_m|>R_k$ and over all $k$ and $m\in \n_k$
such that $|a_m|>R_k$, by using the localization conditions on $g_n$ and 
$\widehat{f_n}$.  

\vs

\textbf{Step 6. Comparison of the densities. } 
Finally we combine the inequality obtained in Step~4 with
(\ref{eq:st5f}) and similar inequalities with other combinations of
indices. Then  we obtain  for every $n\in \n(R-r)$ 
\[
\sum_{m\not\in\n(R)}(|c_m^n||d_m^n|+|c_n^m||d_n^m|)\le
\Big(C_1r^{2-2s}+C_2\sum_{k=0}^\infty N(R_{k+1})(R_k-R+r/2)^{-2s}\Big)^{1/2},
\] 
where $C_1$ and $C_2$ depend on $S,T,B,s$, and  $\sup_n\|f_n\|_2$ and $\sup_n\|g_n\|_2$.
Assumption $(c)$ (the estimate of the upper density $D^+(\Lambda )
<\infty $) implies that 
\[N(R_{k+1})\le D_12^{2d(k+1)}R^{2d}\quad{\rm and}\quad N(R-r)\le D_1(R-r)^{2d}\] 
for some $D_1>0$ and all $R$ large enough. Then for $R$ large enough we obtain
\begin{multline*}
\sum_{n\in\n(R-r)}\sum_{m\not\in\n(R)}(|c_m^n||d_n^m|+|c_n^m||d_m^n|)\le\\
N(R-r)\Big(C_1r^{2-2s}+C_2 D_12^{2d}R^{2d}(r/2)^{-2s}+C_2D_1R^{2d}\sum_{k=1}^\infty 2^{2d(k+1)}(2^{k}R-R+r/2)^{-2s}\Big)^{1/2}\\
\le 
C (R-r)^{2d}R^{d}r^{-s},
 \end{multline*}
 where $C$ depends on $S, T, B, s, D_1, d$ as well as $\sup_n\|f_n\|_2$ and $\sup_n\|g_n\|_2$. 
 
 To finish the proof we recall  that   $r=R^\delta$ and
 $\delta\in(d/s,1)$. 
 Observe  that for
 $r$ large enough the estimate of the upper density implies
 $N(R)-N(R-r)\le D_1R^{2d-1}r$ (we just cover the set $Q(0,R)\setminus
 Q(0,R-r)$ by cubes with side length $r$). Now, combining
 (\ref{eq:s}), (\ref{eq:Rr}), and the last inequality,  we obtain 
 \[
 N(R)\le C_3\Big(N(R)-N(R-r)+R^{3d-\delta s}\Big)\le C_4(R^{2d-1+\delta}+R^{3d-s}).\]
If we now let $R$ go to infinity, we see that
$$ 
D^-(\Lambda ) \leq \lim _{R\to \infty } \frac{N(R)}{R^{2d}} \leq C_4
\lim _{R\to \infty } \Big( R^{\delta -1}+ R^{d-s} \Big)  = 0 \, .
$$ 
This conclusion  contradicts the assumption $(c)$ that the   lower
density estimate $D^-(\Lambda )$ is strictly positive.  
 \end{proof}
 
 Lemma~\ref{l:un}  concludes the proof  of  Theorem \ref{th:1} for
 the case of an orthonormal basis. For a Riesz basis we are not able
 to prove that the \tf\  localization of magnitude $s$ (condition (a) of the lemma) with $s>d$ implies the required localization for its
 biorthogonal basis (stated  in (b)). For the general case a more
 complicated argument is presented in  the next sections.   
 
\section{Some preliminaries on modulation spaces}

The proof of Theorem~\ref{th:1} for Riesz bases requires some tools
from \tfa . We give a minimalistic  account of the required  facts on the  short-time
Fourier transform and modulation spaces. The reader can find  the
details and a  much more general
theory of modulation spaces in  \cite{G1}. 

\vs

\textbf{Short-time Fourier Transform.} 
For $(a,b) \in \rdd $ we write 
$$
\pi (a,b) f(t) = e^{2\pi i b\cdot t} f(t-a)
$$
for the \tfs\ of a function $f$ on $\rd $.   
Let $g(x)=2^{-d/4}e^{-\pi|x|^2}$ be the normalized Gaussian function
on $\bR^d$. We consider the short-time Fourier transform of a function
$\phi\in L^2(\bR^d)$ with respect to $g$, \cite[Chapter 3]{G1} 
\[
V_g\phi(x,\xi)= ( \phi  , \pi (x,\xi )g ) =
\int_{\R^d}\phi(t)\overline{g(t-x)}e^{-2\pi it\cdot \xi}dt,\quad
x,\xi\in \R^d.\]  The inversion formula for the \stft\ yields 
\[
\phi(t)=\iint_{\R^{2d}}V_g\phi(x,\xi)e^{2\pi i\xi\cdot t}g(t-x)dx,
\]
for every $\phi\in L^2(\R^d)$, with a weak interpretation of the
vector-valued integral. 

\vs

\textbf{Modulation Spaces.}
For each $s\ge 0$ let 
\[L^2_s(\R^m)=\{f\in
L^2(\R^m):\|f\|_{L_s^2}^2=\int_{\R^m}|f(x)|^2(1+|x|)^{2s}dx<\infty\}\, .\] 
The
modulation space $M^2_s(\rd )$ is defined by 
$$M_s^2(\rd ) =\{\phi\in L^2(\R^d):
\|\phi \|_{M^2_s}=\|V_g\phi\|_{L^2_s(\rdd )}<\infty\}.
$$
 The following
norm equivalence identifies the \modsp\ $M^2_s$ with the
Fourier-Lebesgue space $L^2_s \cap \cF L^2_s$, see~\cite[Prop.~11.3.1,
12.1.6]{G1}:  
\[c_1\|\phi\|_{M_s^2}\le \|\phi\|_{L_s^2}+\|\widehat{\phi}\|_{L_s^2}\le c_2\|\phi\|_{M_s^2}.
\]

 The adjoint  operator $V_g^*$ of the \stft\ is  defined on $L^2(\R^{2d})$ by
\[
V^*_gF(t)=\iint_{\R^{2d}}F(x,\xi)e^{2\pi i\xi\cdot t}g(t-x)dxd\xi =
\iint _{\rdd } F(x,\xi ) \pi (x,\xi )g (t) dxd\xi \,  .\]
If $F\in L^2_s(\rdd )$,  then by~\cite[Prop.~11.3.2]{G1} $V^*_gF\in M^2_s(\rd )$ and 
\begin{equation}
\label{eq:V*}
\|V^*_gF\|_{M^2_s}\le C\|F\|_{L^2_s}.
\end{equation}

We note that the phase-space localization of magnitude $s$ 
can be rephrased as 
$$
\sup _{n\in\bN}\inf _{(a,b) \in \rdd } \|\pi (a,b) f_n \|_{M^2_s} <\infty \,
.
$$


\textbf{Amalgam spaces.} 
We define the amalgam space  $W(L^2_s)\subset L^2_s (\rdd ) \cap
L^\infty (\rdd )$ as the space of all continuous function on $\rdd $
for which the  norm 
\[\|F\|^2_{W(L^2_s)}=\sum_{k,n\in \zd
} \sup_{x,\xi\in[0,1]^d}|F(x+k,\xi+n)|^2\, (1+|k|+|n|)^{2s}\]
is finite. For $s=0$, $\|F\|_{W(L^2)}\ge \|F\|_2$ obviously.
The continuity of $F$ implies the existence of points $x_{kn}, \xi
_{kn} \in [0,1]^d$, such that 
\begin{equation}
  \label{eq:c9}
\|F\|^2_{W(L^2_s)}=\sum_{k,n\in \zd }|F(k+x_{kn}, n+\xi
_{kn})|^2(1+|k|+|n|)^{2s}\, .
\end{equation}
The definition of $W(L^2_s)$ implies the  following sampling
inequality: If     $\Lambda =
\{\lambda _n\} \subseteq \rdd  $ is  relatively separated, $z\in \rdd
$,  and $F\in W(L^2_s)$, then 
\begin{equation}
  \label{eq:c18}
\Big( \sum _{n} |F(z+\lambda _n)|^2 \, (1+|z+\lambda
_n|)^{2s}\Big)^{1/2} \leq \sup _{k\in \zdd } \card \, \big(\Lambda
\cap 
(k+[0,1]^{2d})\big) \, \|F\|_{W_{L^2_s}} \, .
\end{equation}
The following important inequality links \modsp s with amalgam spaces: 
For every $\phi\in L_s^2(\rd) $ with $\widehat{\phi}\in L_s^2(\rd)$ we have,
e.g., by \cite[Theorem 12.2.1]{G1},
\begin{equation}
  \label{eq:c32}
  \|V_g\phi\|_{W(L^2_s)}\le C\|V_g\phi\|_{L_s^2} = C \|\phi \|_{M^2_s} 
\, .
\end{equation}

\section{Basis modification}

To finish the proof of Theorem~\ref{th:1}, we will modify a given \rb\
$\{f_n\}_n$  for $L^2(\R^d)$ 
that has phase-space localization of magnitude $s>d$ into a \rb\  with  much better localization
properties. The argument in this section may be of independent
interest and can also be  used to prove positive results about frames
and bases.

\begin{proposition}\label{improve} 
Assume that $\{f_n\}_{n=1}^\infty$ is a Riesz basis for $L^2(\R^d)$ that satisfies
 \begin{equation}
 \label{eq:ab}
\sup _{n\in \bN }\Big( \int_{\R^d}|x-a_n|^{2s}|f_n(x)|^2+ 
\int_{\R^d}|\xi -b_n|^{2s}|\widehat{f_n}(\xi)|^2 \Big) <\infty
\end{equation}
for some $s>d$. 
Then there exists a Riesz basis $\{ h_n\}_{n=1}^\infty$ that satisfies   
 \begin{equation}
 \label{eq:abc}
\sup _{n\in \bN }\Big( \int_{\R^d}|x-a_n|^{2t}|h_n(x)|^2+ 
\int_{\R^d}|\xi -b_n|^{2t}|\widehat{h_n}(\xi)|^2 \Big) <\infty
\end{equation}
for \emph{every } $t>0$. 
\end{proposition}

\begin{proof}
  The new basis is obtained by a  modification of $\{f_n\}_n$.   We use  the inversion formula for 
the short-time Fourier transform and truncate it. In the language of
\tfa\ we apply a localization operator to $f_n$. Precisely, let $R>0$
and  $Q(R)=Q(0,R)=B(0,R)\times B(0,R)\subset\R^{2d}$.
Then the localization operator $A_R$ is defined by 
  \[
  A_Rf(t) = \iint_{Q(R)}V_g f(x,\xi)e^{2\pi i\xi\cdot t}g(t-x)dxd\xi,\quad f\in L^2(\R^d)\, .
\]
Intuitively, $A_Rf$ is the part of $f$ that is concentrated on the set
$Q(R)$ in the phase space. For more on localization operators
see for instance~\cite{Wong1999,CG03}. 

We recast the assumption as follows: 
 $\{f_n\}_n$ is  a Riesz basis for $L^2(\R^d)$,  $f_n(x)=e^{2\pi i b_n x}\phi_n(x-a_n)$, $s>d$ and
\[\sup_{n\in\N}\|\phi_n\|^2_{M_s^2}\le S<\infty.\]
 We  now define 
\begin{equation}
\label{eq:hn}
\psi_n= A_R\phi _n  = \iint_{Q(R)}V_g\phi_n(x,\xi)\pi (x,\xi )g\,
dxd\xi \, 
\end{equation}
and the modified basis $h_n (x) = h_n^{R}(x) = e^{2\pi i b_n x}\psi_n(x-a_n)$. 

\vs

\textbf{Claim. }  For $R$ large enough $\{h_n\}_{n=1}^\infty$ 
is a Riesz basis for $L^2(\R^d)$. 


To prove the claim,  
 it suffices to
show that for every $\epsilon>0$ there exists $R$ such that for $h_n=h_n^{R}$
 and every sequence $\{c_n\}_n\in \ell ^2$ the inequality 
\[
\Big\|\sum_n c_n(f_n-h_n)\Big\|_2\le \epsilon\|c_n\|_2\]
holds. If $\{f_n \}_n$ is a Riesz basis with the lower basis constant $A>0$, then 
$$
\|\sum _n c_n h_n\|_2 \geq \|\sum _n c_n f_n\|_2 - \|\sum _n c_n (f_n
- h_n)\|_2 \geq (A-\epsilon ) \|c\|_2 \, ,
$$
and so $\{h_n\}_n$ is a Riesz basis. 

Using once again the crucial assumption $s>d$, we now choose a number
$\sigma $ such that $d< \sigma < s$. Using the inversion formula for
the \stft\ 
and~\eqref{eq:hn}, we write 
$$
\phi _n - \psi _n = \intrdd \big(1-\chi _{Q(R)}(x,\xi )\big) V_g \phi _n (x,\xi
  ) \pi (x,\xi)g\, dxd\xi \, ,$$
and estimate the $M^2_\sigma $-norm of $\phi_n -\psi _n$ with 
 (\ref{eq:V*}) as 
 \begin{align*}
\|\phi_n-\psi_n\|_{M^2_{\sigma }}^2& \le \intrdd  \big( 1 - \chi_{ Q(R)}
  (x,\xi )    \big) | V_g\phi_n(x,\xi)| ^2 (1+|x|+|\xi |)^{2\sigma } \, dxd\xi \\
&\leq (1+R)^{2(\sigma -s)} \intrdd  \big( 1 - \chi_{ Q(R)}
  (x,\xi )    \big) | V_g\phi_n(x,\xi)| ^2 (1+|x|+|\xi |)^{2s  } \, dxd\xi \\
&\leq C (1+R)^{2(\sigma -s)} \|\phi _n \|_{M^2_s}^2 \leq C(1+R)^{2(\sigma
  -s)} S \, .
 \end{align*}
Choosing now $R$ large enough, we have 
$$
\|\phi _n - \psi _n \|_{M^2_\sigma } < \epsilon \qquad {\text{for all}}\  n \, .
$$
Then we have, with \eqref{eq:c9} and a suitable choice of points
$(x_{k,m},\xi_{k,m})\in[0,1]^{2d}$, that 
\begin{multline*}
\Big\|\sum_n c_n(f_n-h_n)\Big\|^2_2 
=\Big\|\sum_nc_nV_g(f_n-h_n)\Big\|_2^2\le 
C\Big\|\sum_n c_nV_g(f_n-h_n)\Big\|_{W(L^2)}^2 \\
=\sum_{(k,m)\in\Z^{2d}}\Big|\sum_n c_nV_g(f_n-h_n)(x_{k,m}+k,
  \xi_{k,m}+m)\Big|^2 \\
\leq \sum _{(k,m) \in \zdd } \Big( \sum _n |c_n| \, 
  |V_g(\phi_n-\psi_n)(x_{k,m}+k-a_n, \xi_{k,m}+m-b_n)| \Big)^2\\ 
\le \sum _{(k,m) \in \zdd }\Big(\sum_n 
  (1+|k-a_n|+|m-b_n|)^{-2\sigma}\Big)\, \times \\
\times  \Big(\sum _{n}|c_n|^2  |V_g(\phi _n -
\psi _n)(x_{k,m}+k-a_n, \xi_{k,m}+m-b_n)|^2
(1+|k-a_n|+|m-b_n|)^{2\sigma }\Big) \, .
\end{multline*}

Since $\sigma >d$ and $\Lambda $ is relatively separated,  the sum
$\sum_{n}   (1+|k-a_n|+|m-b_n|)^{-2\sigma}$ is uniformly bounded
independent of $k$ and $m$. 
Thus we obtain
\begin{multline*}
\Big\|\sum_n c_n(f_n-h_n)\Big\|^2_2\le \\
\le C\sum_n |c_n|^2\sum_{k,m}|V_g(\phi _n -
\psi _n)(x_{k,m}+k-a_n, \xi_{k,m}+m-b_n)|^2
(1+|k-a_n|+|m-b_n|)^{2\sigma }\, .
\end{multline*}
By~\eqref{eq:c18} and~\eqref{eq:c32} we estimate further that 
\begin{align*}
\sum _{k,m}  |V_g(\phi _n -
\psi _n)(& x_{k,m}+k-a_n, \xi_{k,m}+m-b_n)|^2
(1+|k-a_n|+|m-b_n|)^{2\sigma } \\
& \leq C_2 \|V_g(\phi _n - \psi
_n)\|_{W(L^2_\sigma )}^2 \leq C_3\|\phi _n - \psi _n\|_{M^2_\sigma }^2 <
C_3 \epsilon ^2 \, .
\end{align*}

Collecting all estimates, we arrive at
$$
\Big\|\sum_n c_n(f_n-h_n)\Big\|^2_2 \leq 
C_3\epsilon^2\sum_n |c_n|^2
=C_3 \|c\|_2^2 \, \epsilon ^2 \, .
$$

Consequently, $\{h_n\}_n$ is a \rb\ for $\lrd $.  

Finally, applying (\ref{eq:V*}) once again, we obtain,  for arbitrary
$t>0$, 
\[
\|\psi _n\|_{L^2_t} + \|\hat{\psi } _n\|_{L^2_t} \leq C
\|\psi_n\|_{M^2_{t}}\le C'\|V_g\phi_n\chi_{Q(R)}\|_{L^2_{t}}\le C_t 
R^{t}\, ,\]
which is \eqref{eq:abc}.
\end{proof} 

\rem\  The construction of $\psi _n$ implies that $|\psi _n(t)
| \leq C e^{-\alpha |t|^2}$ and $|\widehat{\psi _n}(\xi  ) | \leq C
e^{-\beta |\xi |^2}$ for some $\alpha , \beta , C >0$. Thus the
perturbed basis belongs to the Gelfand-Shilov space $S^{1/2,1/2}$, the
smallest space of test functions that is invariant under the Fourier
transform. 

\vs


To complete the proof of Theorem \ref{th:1} we will show
that the biorthogonal basis $\{\widetilde{h_n}\}_n$
satisfies~\eqref{eq:ab}  for some $s$ large enough and then  apply Lemma \ref{l:un}.

The modified basis $\{h_n\}_n$ possesses enough phase-space localization so
that the theory of localized frames ~\cite{FoG05,G2} is
applicable. We say that a frame $\{h_\lambda : \lambda \in\Lambda \}$
is $s$-localized over the index set $\Lambda \subseteq \rd $, if its  Gramian satisfies 
\begin{equation}
  \label{eq:c20}
|(h_\mu , h_\lambda )| \leq C (1+|\lambda -\mu |)^{-s} \qquad \text{
  for all } \lambda , \mu \in \Lambda \, .
\end{equation}
The main result about localized frames asserts that the dual frame
possesses the same type of localization. 
Specifically we need the following result taken from~\cite[Thm.~1.1, Cor.~3.7]{FoG05}

\begin{proposition}
 \label{frameloc}
  Let $\Lambda \subseteq \rd $ be a relatively separated set and
  $\{h_\lambda : \lambda \in \Lambda \}$ be a frame for $\lrd
  $. Assume that $\{h_\lambda \}$ is $s$-localized for $s >d$. Then
  the (canonical) dual frame is also $s$-localized, i.e., 
  \begin{equation}
    \label{eq:43}
|(\widetilde{h_\mu} , \widetilde{h_\lambda } )| \leq C' (1+|\lambda -\mu |)^{-s} \qquad \text{
  for all } \lambda , \mu \in \Lambda \, .    
  \end{equation}
\end{proposition}

For a Riesz basis this result can be proved directly. The Gramian
matrix $G$ of the basis with entries $( h_\mu , h_\lambda )$ is
invertible with inverse $(G\inv )_{\lambda \mu } = (\widetilde{h_\mu}, \widetilde{h_\lambda } )$. By a theorem of Jaffard \cite{Jaf} the
polynomial off-diagonal decay is preserved under inversion, whence  follows the
statement of the proposition (for a \rb ).

To apply Proposition~\ref{frameloc}, we need to compare the phase-space
localization of magnitude $s$ in Bourgain's \up\ with the localization
defined by the off-diagonal decay of the Gramian. 

\begin{lemma} \label{compare}

(i) If  $\{h_n\}_n$ is a \rb\ that has phase-space  localization
of magnitude $s>0$,  then $\{h_n\}_n$ is  $s$-localized
over the index set $\Lambda$ in the sense of \eqref{eq:c20}. 

(ii)   If  $\{h_n\}_n$ is a \rb\ that has phase-space localization of 
magnitude  $s>3d$, then the biorthogonal basis
$\{\widetilde{h_n}\}_n$ 
has phase-space localization of magnitude $t$ for any $t\in (d, s)$.  
\end{lemma}

\begin{proof}
(i)  We choose the set $\Lambda = \{(a_n,b_n)\}_n \subset\R^{2d}$ as the
appropriate index set. By Lemma~\ref{l:1}, $\Lambda $ is relatively
separated. Then 
\begin{multline*}
 | (h_n, h_m)| 
 \leq \\
 \leq
 \sup _{x\in \rd }
  \Big((1+|x-a_m|)(1+|x-a_n|)\Big)^{-s} \, \intrd h_n(x)
  (1+|x-a_n|)^{s} \overline{h_m(x)}  (1+|x-a_m|)^{s}  \, dx
  \\
\leq C(1+|a_m - a_n|)^{-s} \sup _{n\in\bN}\intrd |h_n(x)|^2 
  (1+|x-a_n|)^{2s}\, dx = (1+|a_m - a_n|)^{-s} S^2 \, .    
\end{multline*}
The argument for the Fourier transforms yields that
$$
| (h_n, h_m)| = | (\widehat{h_n}, \widehat{ h_m} )| \leq C(1+|b_m -
b_n|)^{-s} S^2 \, . 
$$
 By combining both estimates we arrive at
$$
| (h_n, h_m)| \leq C (1+|\lambda -\mu |)^{-s} \qquad \text{
  for all } \lambda , \mu \in \Lambda \, ,
$$
in other words, $\{h_n\}$ is $s$-localized over $\Lambda $. 

(ii)  By (i) $\{h_n\}$ is $s$-localized on $\Lambda \subseteq \rdd $
and $s>3d$. Therefore  Proposition~\ref{frameloc} applies and  the dual frame is also
$s$-localized as in~\eqref{eq:43}. Now we have
\[
\widetilde{h_n}=\sum_m(\widetilde{h_n}, \widetilde{h_m})h_m.\]
Next let $d<t<s-2d $ (which is possible by $s>3d$). A straightforward
calculation and (\ref{eq:abc}) give 
\begin{align*}
\int_{\R^d}|x-a_n|^{2t}|h_m(x)|^2dx& \le \int_{\R^d}C_t\big( |x-a_m|^{2t}
+ |a_n - a_m|^{2t} \big) |h_m(x)|^2dx \\
&\leq  C (S + |a_n - a_m|^{2t}) \leq C'(1+|\lambda_n-\lambda_m|)^{2t}.  
\end{align*}
With the triangle inequality for the $L^2_t$-norm we obtain 
\begin{align*}
\Big(\int_{\R^d}|x-a_n|^{2t}|\tilde{h}_n(x)|^2dx\Big)^{1/2}&= \Big(
  \intrd |\sum _m(\tilde{h}_n, \tilde{h}_m)h_m(x)|^2 |x-a_n|^{2t} \,
  dx\Big)^{1/2} \\
&  \le C \sum_m|(\tilde{h}_n,
\tilde{h}_m)|(1+|\lambda_n-\lambda_m|)^{t} \\
&\leq C \sum _m  |(1+|\lambda_n-\lambda_m|)^{t-s} \, .
\end{align*}
Since $\Lambda $ is relatively separated and $t-s < -2d $,  the  last sum is uniformly
bounded.

 Similar estimates hold for the Fourier transforms. Consequently, the
 dual basis is localized in the sense of 
\eqref{eq:ab} for  $t\in (d,s)$. 
\end{proof}

We now can  finish the proof of Theorem~\ref{th:1} for Riesz
bases.

We start with a \rb\  $\{f_n\}_n$ that satisfies \eqref{eq:ab}
for some $s>d$. Then we modify this basis by means of Proposition~\ref{improve} to
a \rb\ $\{h_n\}_n$ that satisfies the localization estimates \eqref{eq:ab} for \emph{all}
$t>0$. Finally, Lemma~\ref{compare} guarantees that the dual basis $\{
\widetilde{h_n} \}_n$ also satisfies the  localization
estimates~\eqref{eq:ab} for \emph{all} $t>0$. Thus all assumptions of
Lemma~\ref{l:un} are satisfied whence we conclude that the
localization of $\{h_n\}_n$ is $t\leq d$. This
is a contradiction to the construction of $\{h_n\}_n$. 
This means that the original basis $\{f_n\}_n$ cannot  satisfy the
strong localization estimate $s>d$, 
and  the proof of Theorem \ref{th:1} is complete. 
 
\section{Sampling in Bargmann-Fock spaces and Concluding Remarks}

We finally give an application of Theorem~\ref{th:1} to  several complex variables.

Recall that the Bargmann-Fock space $\cF $
consists of all entire functions on $\cd $  with norm $\|F\|_\cF ^2 = \int _{\cd } |F(z)|^2 \,
e^{-\pi |z|^2}\, dz$. The Bargmann-Fock space possesses the  reproducing
kernel $K_w(z) = e^{\pi \bar{w} \cdot z}$ for $w,z\in \cd $, so that
$F(w) = (F, K_w)$. Consequently, a set $\{ K_\lambda : \lambda \in
\Lambda \}$ is a \rb\  for $\cF $, \fif\ $\Lambda \subseteq \cd$ is
simultaneously  sampling and  interpolating  for $\cF $, i.e., $\sum _\lambda |F(\lambda )|^2
e^{-\pi |\lambda |^2} \asymp \|F\|_\cF ^2$,  and for every $c\in \ell
^2(\Lambda )$ there exists a (unique) $F\in \cF $, such that
$F(\lambda )e^{-\pi |\lambda |^2/2} = a_\lambda $. 

 The following  result is an immediate consequence of  Theorem~\ref{th:1}.

\begin{theorem}\label{fock}
  The Bargmann-Fock space does not admit a set $\Lambda \subseteq \cd
  $ that is simultaneously  sampling and  interpolating. 
\end{theorem}

\begin{proof}
We use the Bargmann  transform defined as 
\begin{equation*}
 \label{bt}
   \cB f(z)=F(z)= 2^{d/4}e^{-\pi z^2 /2}\int_{\rd} f(t)e^{-\pi
    t\cdot t}e^{2\pi t \cdot z} dt \, \qquad z\in \cd \, ,
  \end{equation*}
and  translate Theorem~\ref{th:1} into a statement of complex
analysis. The Bargmann transform  is unitary from $\lrd $ onto $\cF $ and maps the \tfs s of the
Gaussian $e^{-2\pi i w_2 \cdot x} e^{-\pi (x-w_1)^2}= \pi (w_1,-w_2)g(x)$ to the
reproducing kernel $ e^{i\alpha} K_w  \, e^{-|w|^2/2}$ for some phase
factor $|c|=1$. Thus $\Lambda \subseteq\cd  $ is a set of sampling and
interpolation, \fif\ $\{K_\lambda : \lambda \in \Lambda \}$ is a Riesz
basis for $\cF $. 

Clearly, the Gaussian satisfies the localization
condition~\eqref{eq:ab} for all $s>0$. By  Theorem~\ref{th:1} a set of
\tfs s of the Gaussian cannot from a \rb\ for $\lrd $. Consequently,
no set $\Lambda \subseteq \cd $ can be simultaneously  sampling and interpolating.  
\end{proof}

\rem\ Theorem~\ref{fock} is a statement of complex analysis. Indeed,
in dimension $d=1$ is it well-known and can be proved with classical
methods from complex variables. 
In higher dimensions the result was expected, but seems to have been
open so far. 
A proof with different methods  has  been
announced  in~\cite{AK}. 


\vs

\textbf{Final remarks.}
It is interesting to compare the critical value of the localization
parameter  $s$ in higher
dimensions with the higher dimensional versions of the Balian-Low
theorem.   
It is known that for $d=1$ and every $s<2$ there exists a function $f$ with
\[
\int_{\R} |x|^s|f|^2<\infty,\quad\int_{\R}|\xi|^s|\hat{f}|^2<\infty,\]
such that $\{e^{2\pi i nt}f(x-m)\}_{n,m\in \Z}$ is an orthogonal basis
for $L^2(\R)$. (A more precise result is obtained in \cite{BCGP}, we
refer the reader to \cite{Jan} also.) Thus in one-dimensional setting
the restrictions on the localization properties of an arbitrary
orthogonal basis and of a Gabor system can be observed only at one
point of our scale, $p=2$. The situation changes drastically when we
consider higher dimensional spaces. There exists an orthogonal basis
$\{f_n\}$ for $L^2(\R^d)$  that satisfies 
\[      
\sup _{n\in \bN } \Big( \inf_{a\in \rd }\int_{\R^d}|x-a|^{2d}|f_n(x)|^2  +
\inf_{b\in \rd} \int_{\R^d}|\xi-b|^{2d}|\widehat{f_n}(\xi)|^2 \Big) <\infty, 
\]
but for every orthogonal basis of the form $\{e^{2\pi i
  nt}f(x-m)\}_{n,m\in \Z^d}$ a multidimensional version of Balian-Low
theorem (see \cite{CP} and references therein) implies 
\[
\sup _{n\in \bN }\inf_{a\in \rd } \int_{\R^d}|x-a|^2|f_n(x)|^2=\infty\quad{\rm{or}}\quad
\sup _{n\in \bN }\inf_{b\in \rd } \int_{\R^d}|\xi-b|^2|\widehat{f_n}(\xi)|^2=\infty.\]
The reason for it could lie in the choice of the radial  weight
$v_p=|x|^p$. It seems that the product weight 
$w_s(x)=\Big((1+|x_1|)...(1+|x_d|)\Big)^s$ might be a
more natural weight in higher dimensions. 

There are  a $(p,q)$-version of Bourgain's theorem~\cite{BP} and a
$(p,q)$ version of the Balian-Low theorem~\cite{Ga} for
$1/p+1/q=1$. It would be interesting to obtain a $(p,q)$-version of
Theorem~\ref{th:1}. 

\vs

\small{\textbf{Acknowledgment.} K.~G.~ would like to thank the Department of
Mathematics of NTNU Trondheim for its hospitality and great 
research environment  during the research  on this paper. Both
authors would like to thank the Centre de Recerca 
Matem\`atica 
 Barcelona for  ideal conditions to complete this  work. }

 \bibliographystyle{abbrv}

\end{document}